\documentclass[12pt]{amsart}
\usepackage{latexsym,fancyhdr,amssymb,color,amsmath,amsthm,graphicx,listings,comment}
       
\newtheorem{lemma}{Lemma}       \newtheorem{coro}{Corollary}
 
\usepackage[section]{placeins}
\let\paragraph\subsection
\title{Weierstrass elliptic functions for the pendulum}
\author{Oliver Knill}
\date{6/21/2023}
\address{Department of Mathematics \\ Harvard University \\ Cambridge, MA, 02138 }

\begin{document}

\begin{abstract}
The mathematical pendulum is traditionally solved using a Jacobi elliptic functions.
We solve it here using the Weierstrass elliptic function. 
Every initial condition of the pendulum 
produces an elliptic curve and a point which by the dynamics of the pendulum 
is translated linearly on the torus. 
\end{abstract}

\maketitle

\section{A polynomial differential equation}

\paragraph{}
The {\bf mathematical pendulum} is the second order nonlinear differential equation 
$$  \theta''= 4c\sin(\theta) \; . $$
The constant $c$ is a real parameter which in a physical setup would be
$c=-g/(4l)$, where $g$ is the gravitational strength and $l$ is 
the length of the pendulum rod.
The pendulum has the {\bf conserved energy} $E=\theta'^2/2+4c \cos(\theta)$
and is so a {\bf Hamiltonian system} $x'=H_y,y'=-H_x$ for the {\bf Hamiltonian}
$$ H(x,y)=\frac{y^2}{2}+4c\cos(x) $$ 
on the cylinder $\mathbb{T} \times \mathbb{R}$. 
Traditionally, the solution is obtained by solving the energy equation for $\theta'$
then separate the variables $t,\theta$ and inverting a {\bf Jacobi elliptic integral}.
We proceed differently using {\bf Weierstrass elliptic curves}.

\paragraph{}
We start by setting $u=e^{i \theta} \in \mathbb{C}$. This produces a
{\bf polynomial differential equation} of second order, where 
$u''= \frac{d^2}{dt^2} u$ is the second derivative for a real or complex
time $t$. 

\begin{lemma}[Polynomial Pendulum Equation]
$u'' = 6cu^2 - 2Eu +2c$. 
\end{lemma}

\begin{proof}
From $u' = i u \theta'$ and using $\theta'^2 u=2Eu-4c(u^2+1)$
and $i \theta'' u = i 4 c u \sin(\theta) = 2c (u^2-1)$, we have
\begin{eqnarray*}
u'' &=& -u \theta'^2 +i u \theta'' \\
    &=& -2Eu+4c(u^2+1) + 2c (u^2-1) \\
    &=& 6cu^2-2Eu+2c  \; . 
\end{eqnarray*}
\end{proof}

\paragraph{}
For $u=w+E/(6c)$ we an depress the quadratic polynomial to the right and get
$w'' = 6cw^2+a$, where $a=2c-E^2/(6c)$. 
With the new variable $p=w/c^{1/3}$ and the new time $z=c^{1/3} t$
and $g_2=-2a/c^{2/3}$, we arrive at the differential equation
$$ p'' = 6 p^2 - g_2/2 \; . $$
This differential equation is solved by
$p(z)=\wp(z+g_1,g_2,g_3)$, where the initial conditions
$(\theta(0),\theta'(0))$ defines the {\bf elliptic curve} 
$$\wp'^2=4\wp^3-g_2 \wp-g_3 $$ 
and a point $g_1$ on it. As we have $g_2$ given from the energy, 
the constant $g_3$ is obtained from the initial conditions 
$w_3=-\wp'^2(0)+4\wp^3(0)-g_2 \wp(0)$
and $g_1$ is fixed from the initial conditions.

\paragraph{The Weierstrass elliptic function}

\paragraph{}
Any two complex numbers $\omega_1,\omega_2$ which define linearly 
independent real vectors in $\mathbb{R}^2$ define a {\bf lattice}
$\Lambda = \{ n \omega_1  + m \omega_2 \; | \;  n,m \in \mathbb{Z} \}$,
the {\bf 2-torus} $\mathbb{T}^2=\mathbb{R}^2/\Lambda$ 
and the {\bf Weierstrass elliptic function} summing over all non-zero
lattice points $\dot{\Lambda}$:
$$ \wp(z) = \frac{1}{z^2} + \sum_{\lambda \in \dot{\Lambda}} 
   \frac{1}{(z-\lambda)^2}-\frac{1}{\lambda^2} \; . $$

\paragraph{}
The $\wp$ satisfies the differential equation
$$ \wp'(z)^2 = 4\wp(z)^3 -g_2 \wp(z)-g_3 \; , $$
where $g_2,g_3$ are determined by the lattice $\Lambda$ so that
$(\wp(z),\wp'(z))$ are on the {\bf elliptic curve}
$$  y^2=4x^3-g_2 x - g_3 \; . $$
One just needs to realize that 
$\wp'(z)^2 - 4 \wp(z)^3 +g_2 \wp(z)$ is analytic and so constant 
\cite{AhlforsComplexAnalysis}.
The elliptic curve over $\mathbb{C}$ is a 2-torus if the {\bf discriminant}
$\Delta= g_2^3-27 g_3^2$ is non-zero. 
Slightly less well known is that $\wp$ satisfies a second order 
quadratic differential equation: 

\begin{lemma}
For arbitrary complex $g_1,g_2,g_3$, the function
$p(z) = \wp(z+g_1,g_2,g_3)$ is the general solution of the
differential equation $p'' = 6 p^2 - g_2/2$.
\end{lemma}

\begin{proof} 
Differentiate the definition of $\wp$ to get
$\wp'=-2 \sum_{\lambda \in \Lambda} \frac{1}{(z-\lambda)^3}$ and 
$\wp''=6 \sum_{\lambda \in \Lambda} \frac{1}{(z-\lambda)^4}$.
The function $6 \wp^2-\wp''$ has lost its singularity at $0$ so that
$6 (\sum_{\lambda \neq 0 \in \Lambda} (\frac{1}{(z-\lambda)^2})^2-\frac{1}{\lambda^2})
-6 \sum_{\lambda \neq 0 \in \Lambda} \frac{1}{(z-\lambda)^4}$ is 
analytic on the torus and therefore must be constant by Liouville's theorem.
Evaluated at $z=0$, to see the constant is $g_2/2$.
\end{proof} 

\paragraph{}
As for the literature, \cite{SiegelComplexFunctionTheoryI} page 61 derives it by differentiating
$\wp'(z)^2 = 4 \wp(z) -g_2 \wp(z)-g_3$ with respect to $z$ and
dividing by $\wp'(z)$. It also appears as Problem 2.4.1 in \cite{PrasolovSolovyev} and
is subject of Corollary 7.1 of \cite{ArmitageEberlein}.

\paragraph{}
The upshot is that the solution of the mathematical pendulum equation 
can be explicitly written down as $\theta(t) = \arg(a+b \wp(q t+g_1,g_2,g_3))$,
where $a,b,q$ are constants depending on $c$,
and $g_1,g_2,g_3$ depend on $\theta_0,\theta'(0)$ and $c$. 
The pendulum trajectory is a closed path on the 
elliptic curve provided $\Delta \neq 0$. This is an elegant alternative
to the solution given as amplitudes of Jacobi elliptic functions
\cite{ArmitageEberlein}, which computer algebra system get to when 
asking for a solution (like with  {\it $DSolve[x''[t]==4 c Sin[x[t]],x[t],t]$} in
Mathematica). 

\section{Pendulum dynamics in the unitary group}

\paragraph{}
The pendulum equation is related to the Lie group $U(1)$.
This can be generalized 
to the unitary $U(n) = \{ A \in M(n,\mathbb{C}), A^* A = 1_n \}$ if 
the  initial conditions commute
\footnote{Unlike in an earlier version, we need commutativity of the initial conditions.}
Every such group element $A$ can be written as $A=e^{i \theta}$ where $a=i \theta$ are
skew Hermitian matrices $a^*=-a$. 
The same integrates also matrix differential equations $\theta'' = 4c \sin(\theta)$, 
where $\theta,\theta'$ are commuting Hermitian matrices.

\paragraph{}
The computation we just did in the Lie algebra 
$g = \mathbb{R}$ of the compact Lie group $G=U(1)$ 
can be done on any Lie algebra $g$ of the unitary group $G=U(n)$
as the elements of the Lie algebra are given by 
Hermitian matrices $x$ so that $\sin$ preserves the Lie algebra.
We get so integrable systems on a subset of $G=U(n)$.
Write an element in the Lie algebra as $a=i \theta$ and an 
element in the Lie group as $u = e^{i \theta}$, then separate 
$e^{i \theta} = \cos(\theta) + i \sin(\theta)$, each $\theta$ is linear 
space of matrices, the exponential can be understood in a linear 
algebra sense so that $e^{i \theta}$ is a matrix, 
representing the group element in $G$. 
Rules like $\sin'(\theta) = \cos(\theta)$ are satisfied by functional 
calculus because Hermitian matrices are normal. 

\paragraph{}
The differential equation $\theta'' =4 c \sin(\theta)$ in the Lie algebra
produces curves $\theta(t)$ in the vector space of Hermitian matrices
and the energy $E = \theta'^2/2 + 4 c \cos(\theta)$ is
conserved if the initial $\theta,\theta'$ commute. 
The differential equation $\theta''=4c \sin(\theta)$ 
has invariant measure of the Hamiltonian flow is located either on 
an equilibrium point or some $k$-dimensional real torus with $k \leq n$.

\begin{proof}
One can diagonalize the $\theta$ reducing the system to $n$ independent penduli 
or proceed as before: define $u=e^{i \theta}$ $w$ with $u=w+E/(6c)$ 
and $p=w/c^{1/3}$ and the new time $z=t/c^{1/3}$ so that $p''=6 p^2 - g_2/2$, 
where $g_2$ is identified with $g_2 I$ in the matrix algebra $G$. 
This means that $p(z) \in M(n,\mathbb{C})$ satisfies the {\bf Weierstrass differential 
equation}. To every matrix entry $p_{ij}(z) \in \mathbb{C}$ is now associated
an {\bf elliptic curve}. If the initial condition $\theta$ and 
velocity $\theta'$ are such that for all $i,j$, the elliptic curve describing 
$p_{ij}(z)$ is non-degenerate, then the solution curve $p_{ij}(t)$ moves on a straight line in 
a $d$-dimensional torus. Once we have determined $p(t)$, we can get $u(t)=c^{1/3} p(t) + E/(6c)$
and get $\theta(t) = {\rm arg}(u(t))$, where the arg function takes an element in the Lie group
$U(n)$ and associates to it $\theta(t)$ such that $e^{i \theta(t)} = u(t)$. 
\end{proof}

\paragraph{}
For almost all initial condition, the solution is 
dense and Diophantine and survives by {\bf KAM theory}
small perturbations of the Hamiltonian. 
Also, by general principles, arbitrary small perturbations of the Hamiltonian 
can now achieve that there are weakly mixing invariant tori \cite{Kni99}.

\begin{coro}
There are arbitrary small perturbations of the pendulum Hamiltonian in the 
unitary group $U(n)$ for $n \geq 2$ such that the system has 
weakly mixing invariant tori.
\end{coro}

\paragraph{}
If $\theta(0),\theta'(0)$ are simultaneously
diagonalizable, we just have a copy of $1$-dimensional systems.
What is special about the above
differential equation is that $(p(0),p'(0)) \in TU(n)$ determines
for every matrix entry an elliptic curve solving also the matrix equation.
The Weierstrass integration works
if we assume that the initial position and velocity commute and so 
are simultaneously diagonalizable.

\section{Remarks}

\paragraph{}
Any Hamiltonian system in a 2-dimensional phase space $M$ is
{\bf integrable} because the energy function $H$ foliates $M$ into level
curves which are 1-manifolds at non-critical values of $H$. Going beyond
Hamiltonian systems, topology starts to matter for the existence of
non-integrable cases: by Poincar\'e-Bendixon, every differential
equation on the plane (or the compactified sphere) is integrable in
the sense that every flow-invariant measure is supported on the fixed point set
or on a finite collection of periodic orbits. There are vector fields on
$\mathbb{T}^2$ already with weakly mixing invariant measures \cite{CFS,HoKn98}.
The pendulum is special because integrability carries over to the
{\bf Sin-Gordon} partial differential equations
$\theta_{tt}-\theta_{xx} = 4c \sin(\theta)$. One also has integrability for
discrete Sin-Gordon equations $\theta_{tt}-L \theta = 4 c \sin(\theta)$, where
$\theta$ is a function on the vertices of a graph and $L$ the Kirchhoff Laplacian of the
graph. The pendulum can also be seen as a {\bf 2-particle Toda system} and so
a Lax Pair description \cite{Kni93diss}. It seems that the Today system is the right higher
dimensional non-linear generalization of the pendulum. 

\paragraph{}
A long-standing open problem in Hamiltonian dynamics is the
{\bf question of Kolmogorov}, whether there are
examples of Hamiltonian systems with {\bf mixing invariant tori}.
This reportedly had been a motivation for Kolmogorov to develop KAM theory
\cite{Arn97}. We had remarked in 1999 \cite{Kni99}
that it is possible to always get from an invariant KAM torus a
weakly mixing invariant $d$-torus with $d \geq 2$ using arbitrary small perturbations.
This works even in infinite dimensions if an invariant finite dimensional torus 
exists on which one has a quasi-periodic motion. 
The idea was that on two or larger dimensional tori, the linear
ergodic flow can be perturbed so that ergodicity and pure point spectrum 
upgrades to singular continuity and so leads to weak mixing. 
Already the flow $\theta'' = 4 c \sin(\theta)$ on the Lie algebra of $U(2)$ 
produces for almost all initial conditions $(\theta,\theta')$ an orbit
that is dense on a $2$-dimensional real torus in $U(2)$. 
A small perturbation of the Hamiltonian can now produce 
solution curves $z(t)=\exp(i \theta(t)) \in U(2)$ which show a 
weak type of chaos. 

\paragraph{}
Let us end this note with a remark on our motivation. 
A core tool in {\bf smooth ergodic theory} is {\bf Pesin theory} (e.g. \cite{KatokStrelcyn}).
The computation of the metric entropy of a system $x'=F(x)$ requires estimating 
{\bf Lyapunov exponents}, the exponential growth rates of the linearization
$v'(t) = dF(x(t)) v(t)$ along an orbits $x(t)$ of the system. 
Herman's pluri-subharmonicity tools \cite{Her83} 
looks promising for measuring the growth rate. But it requires that the system 
is analytic and that the dynamics can be extended analytically to polydiscs. 
Rewriting the pendulum as an analytic system as in Lemma~1 is an attempt 
to make analytic tools kick in. One of the 
systems which could have mixing invariant absolutely continuous measures
is a time periodic perturbation of the standard pendulum. 
The above reduction of the pendulum to a polynomial differential 
equation emerged from such attempts. One could hope for example that some
analytic perturbation of the pendulum system in $U(2)$ allows 
via subharmonicity to establish positive Lyapunov exponents on a set of positive 
measure for an invariant manifold and so Bernoulli components of positive measure
and positive entropy. One has examples of mixing systems like the geodesic flow
on a compact negatively curved manifold but no system with true 
coexistence \cite{Str89} is known, where part of the phase space has quasi-periodic
motion on submanifolds of total positive volume and part of the phase space 
has mixing components of positive volume. 

\section*{Appendix: some illustrations}

\paragraph{}
It is interesting to look at the explicit shape of the elliptic curves. The lattice
must have the property that the translation in the real direction is periodic. 
When the trajectory passes through the pole, we are in the critical case of the 
pendulum or the geometry in that the discriminant is zero.
We got the lattices numerically. The software of course just picks a representative of the lattice. 

\paragraph{}

\begin{figure}[!htpb]
\scalebox{0.5}{\includegraphics{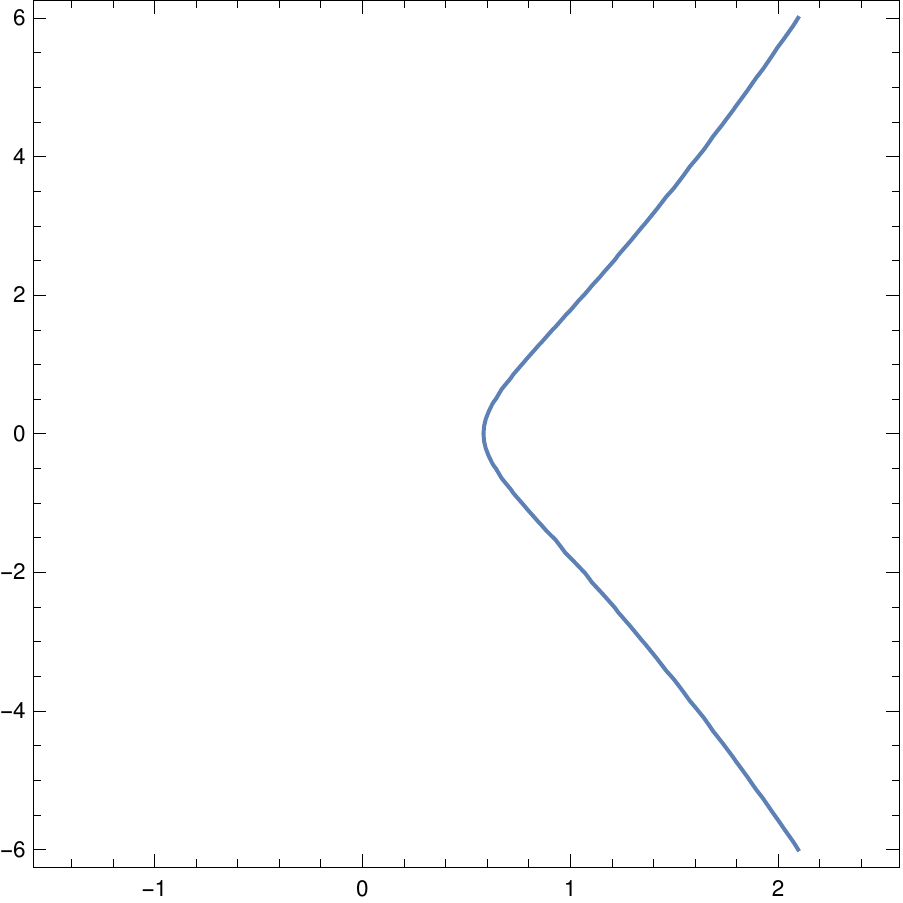}}
\scalebox{0.5}{\includegraphics{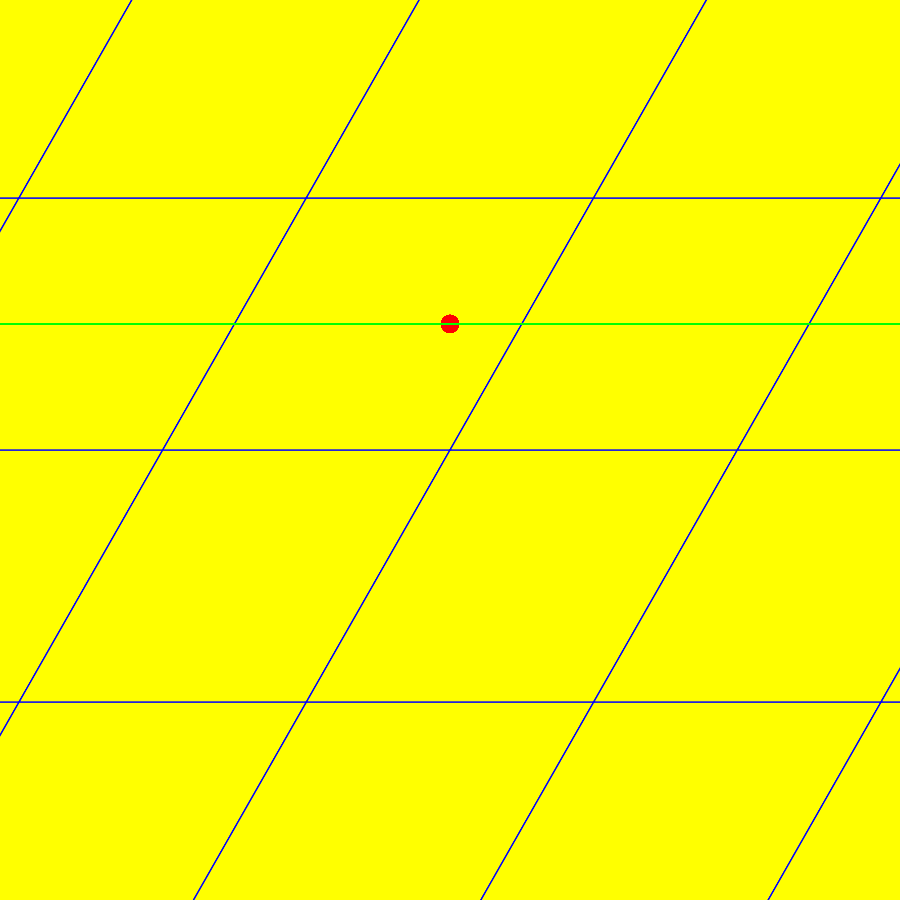}}
\caption{
For $c=-1$ and initial condition $\theta_0=0.$, and $\theta'_0=1.$,
we get the elliptic curve $y^2=4x^3-g_2 x-g_3$ with $g_2= 0.0833333$,
and  $g_3=0.74537$.The lattice is given by $\omega_1= 3.19248\, +0. i$,
and  $\omega_2=1.59624\, +2.80121 i$. The initial point on the curve is
$g_1= 0.\, +1.4006 i$. The discriminant is $\Delta= -15.$.
}
\end{figure}

\begin{figure}[!htpb]
\scalebox{0.5}{\includegraphics{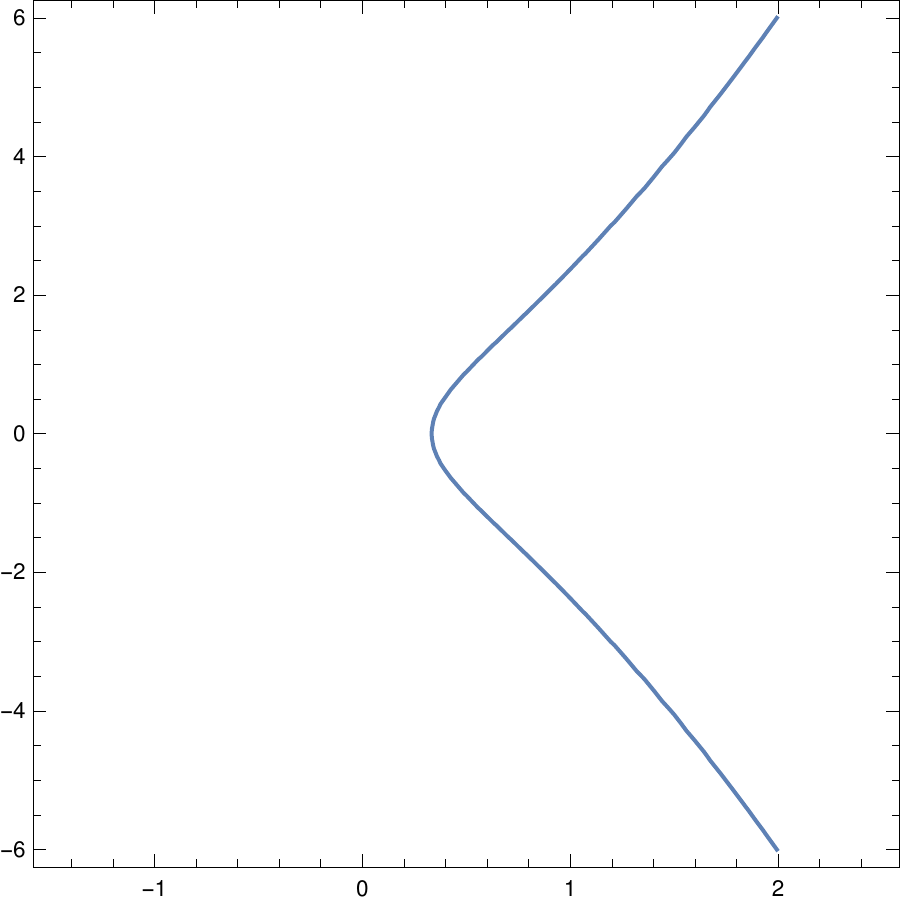}}
\scalebox{0.5}{\includegraphics{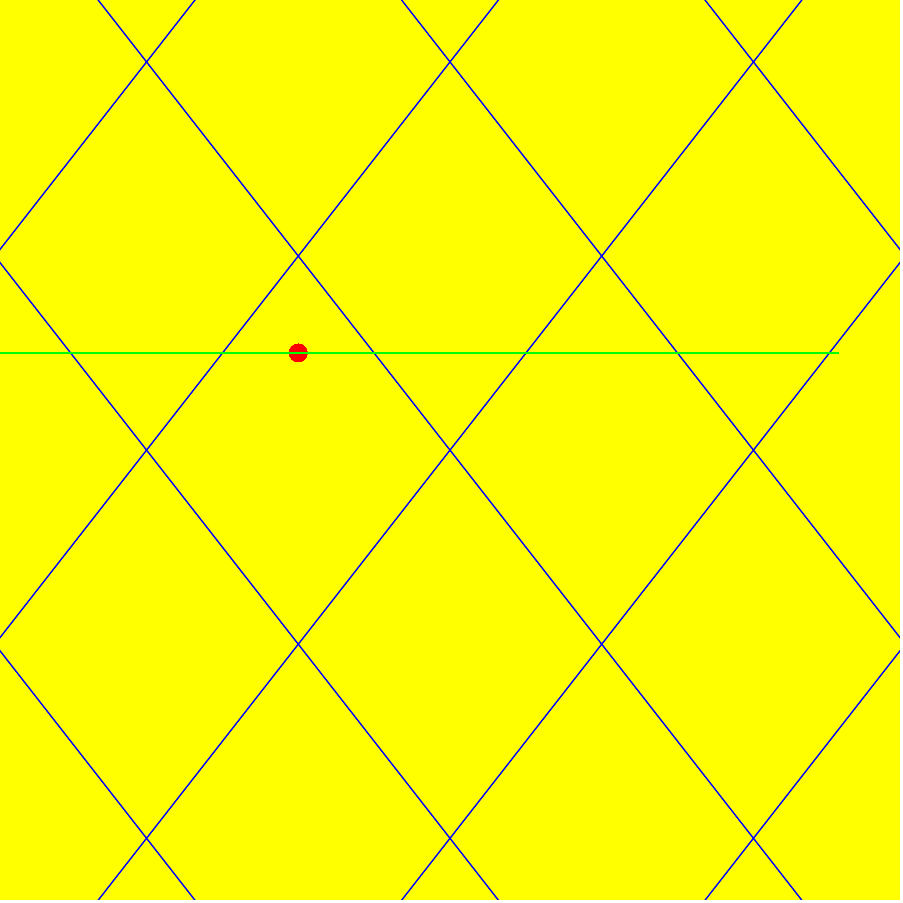}}
\caption{
For $c=-1$ and initial condition $\theta_0=0.$, and $\theta'_0=2.$, we
get the elliptic curve $y^2=4x^3-g_2 x-g_3$ with $g_2= -2.66667$, and
$g_3=1.03704$.The lattice is given by $\omega_1= 1.68575\, +2.15652 i$,
and  $\omega_2=-1.68575+2.15652 i$. The initial point on the curve is
$g_1= -1.68575+1.07826 i$. The discriminant is $\Delta= -48.$.
}
\end{figure}

\begin{figure}[!htpb]
\scalebox{0.5}{\includegraphics{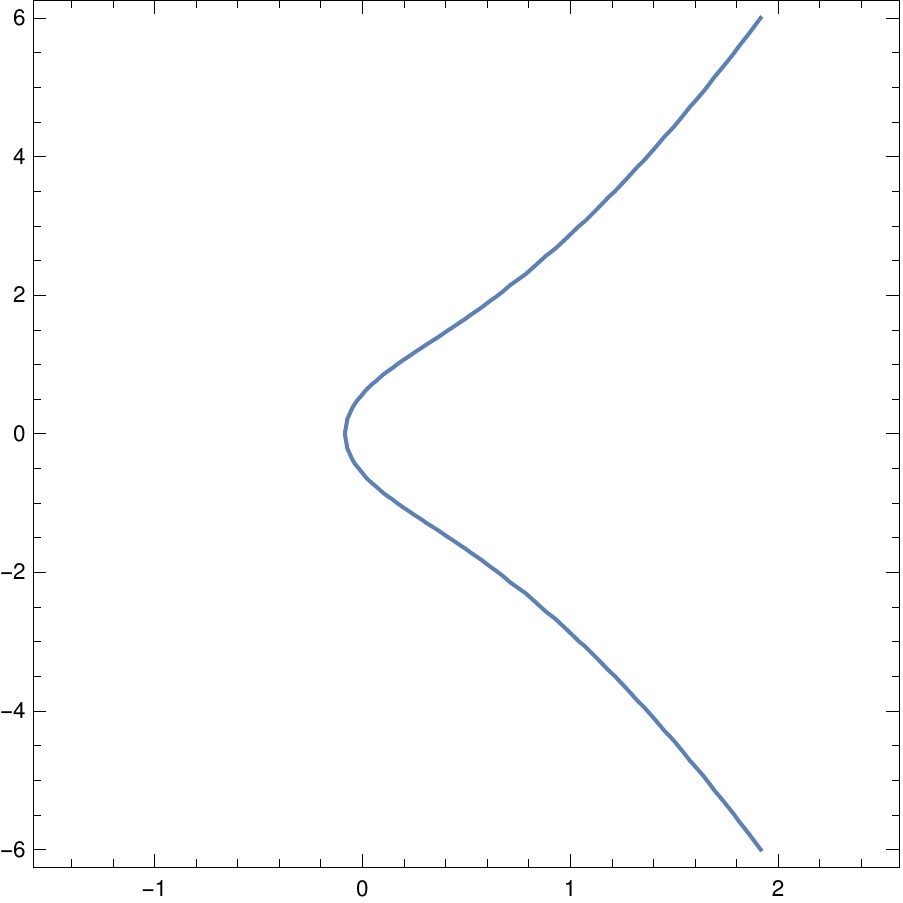}}
\scalebox{0.5}{\includegraphics{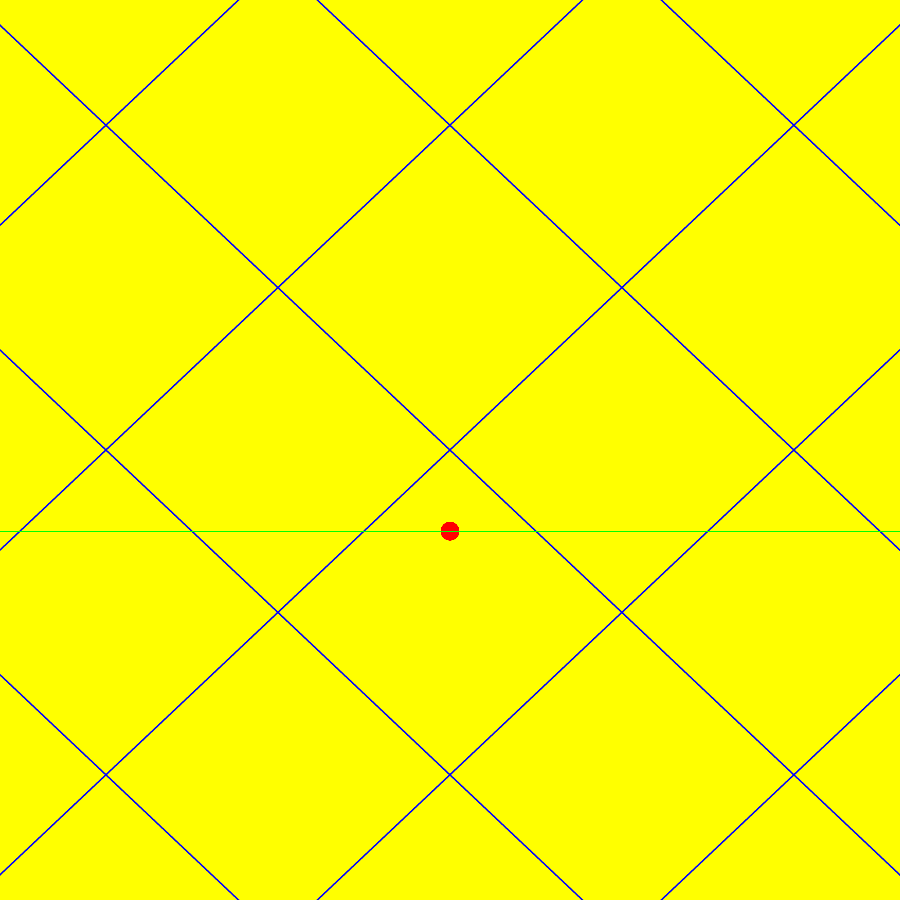}}
\caption{
For $c=-1$ and initial condition $\theta_0=0.$, and $\theta'_0=3.$, we
get the elliptic curve $y^2=4x^3-g_2 x-g_3$ with $g_2= -3.91667$, and
$g_3=-0.328704$.The lattice is given by $\omega_1= 1.91099\, -1.80446 i$,
and  $\omega_2=1.91099\, +1.80446 i$. The initial point on the curve is
$g_1= 0.\, -0.902231 i$. The discriminant is $\Delta= -63.$.
}
\end{figure}

\begin{figure}[!htpb]
\scalebox{0.5}{\includegraphics{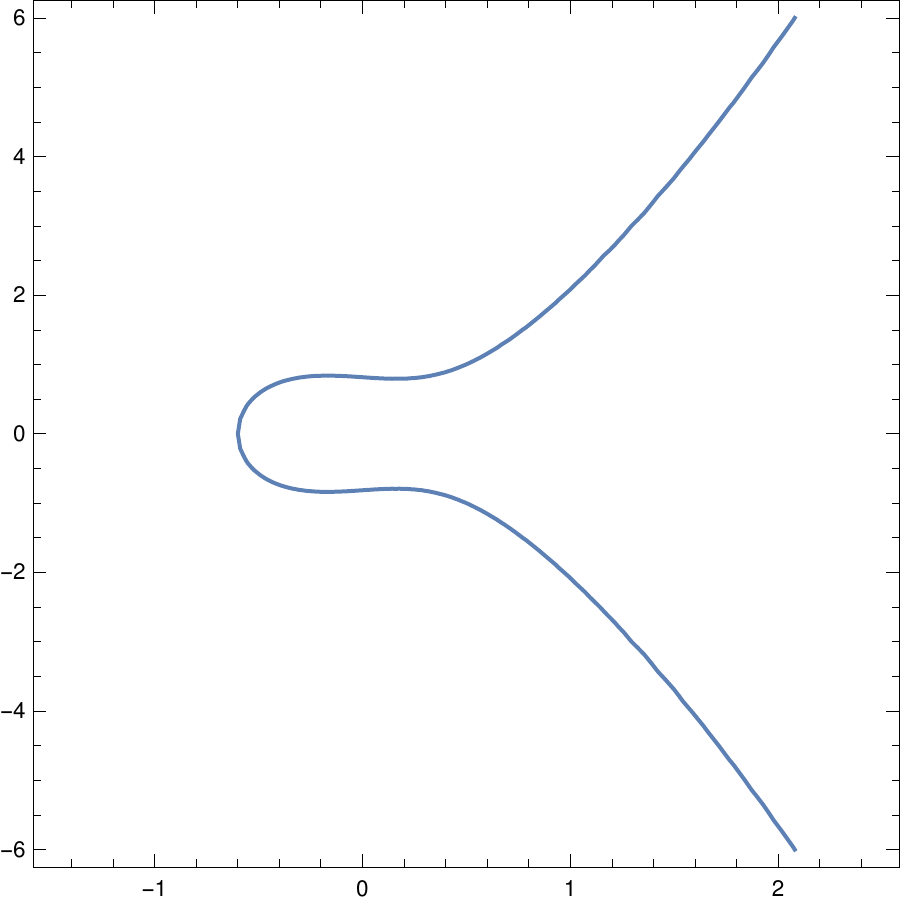}}
\scalebox{0.5}{\includegraphics{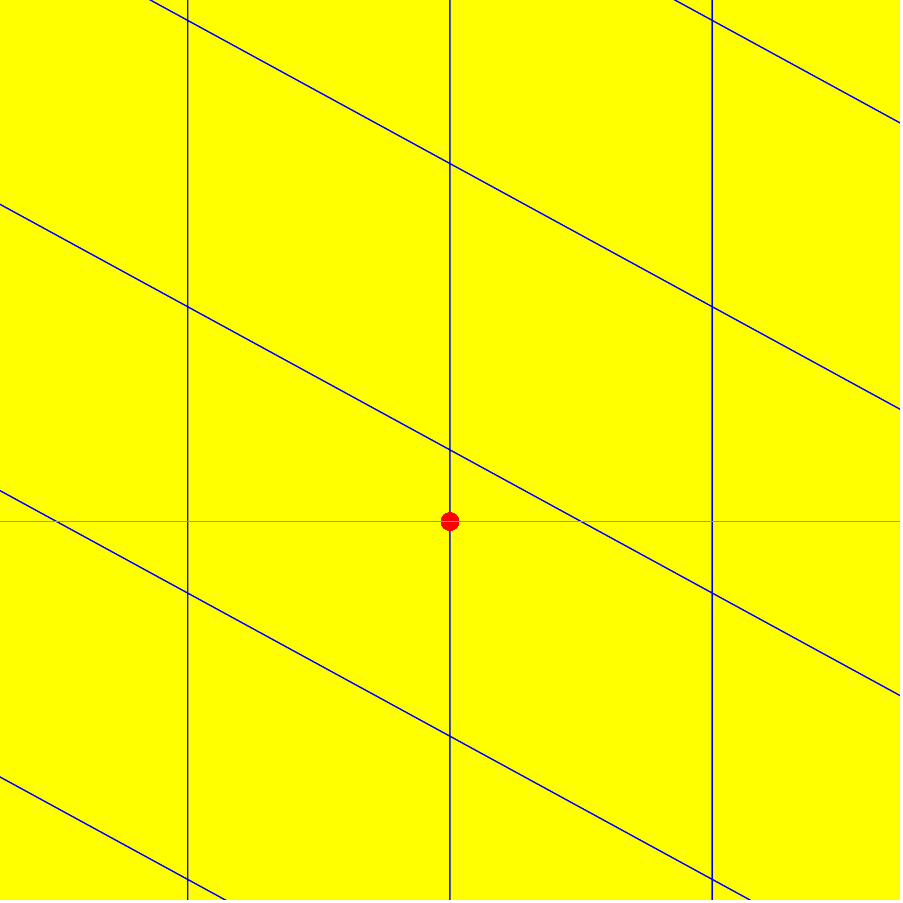}}
\caption{
For $c=-1$ and initial condition $\theta_0=0.$, and $\theta'_0=3.9$,
we get the elliptic curve $y^2=4x^3-g_2 x-g_3$ with $g_2= 0.332008$, and
$g_3=-0.668123$.The lattice is given by $\omega_1= 0.\, -3.18149 i$, and
$\omega_2=2.9144\, -1.59074 i$. The initial point on the curve is $g_1=
0.\, -0.795372 i$. The discriminant is $\Delta= -12.0159$.
}
\end{figure}

\begin{figure}[!htpb]
\scalebox{0.5}{\includegraphics{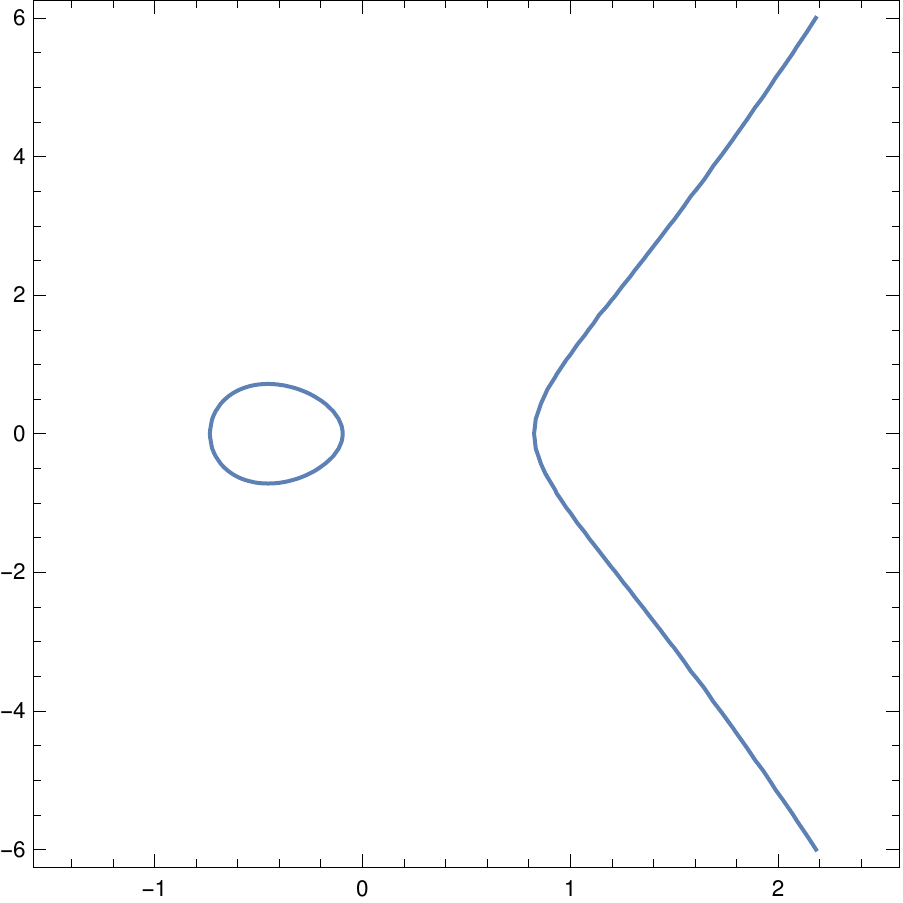}}
\scalebox{0.5}{\includegraphics{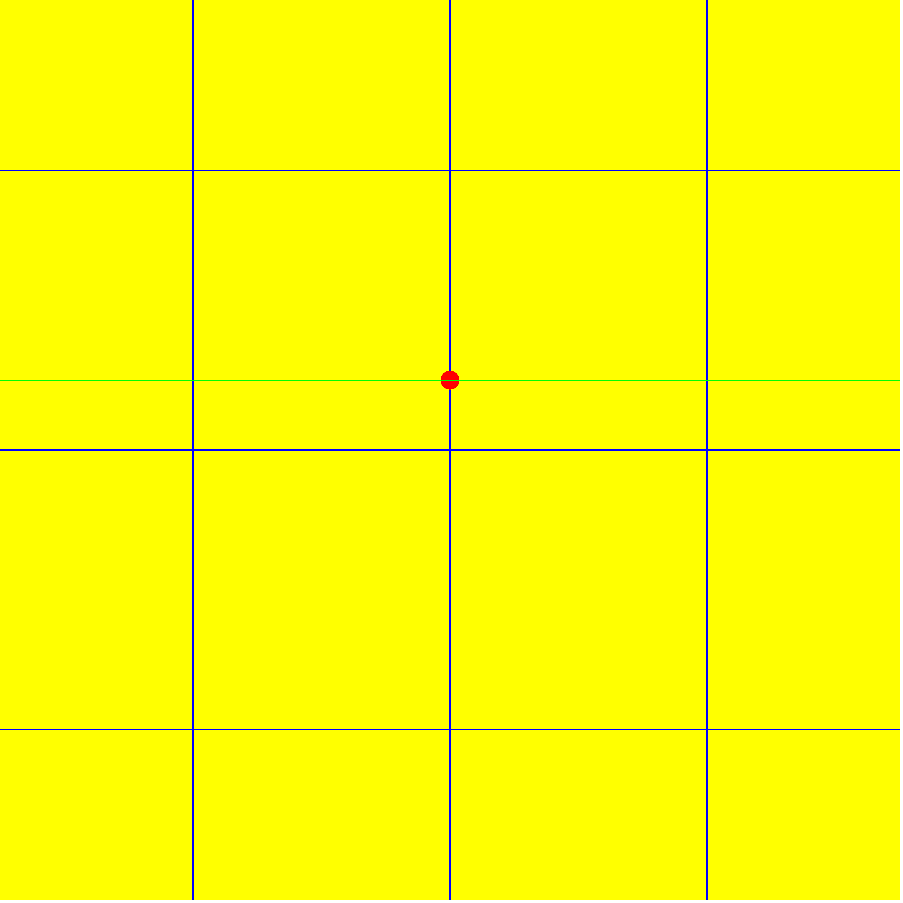}}
\caption{
For $c=-1$ and initial condition $\theta_0=0.$, and $\theta'_0=4.1$,
we get the elliptic curve $y^2=4x^3-g_2 x-g_3$ with $g_2= 2.46801$,
and  $g_3=0.229064$.The lattice is given by $\omega_1= 2.85478\, +0. i$,
and  $\omega_2=0.\, +3.10293 i$. The initial point on the curve is $g_1=
0.\, +0.775731 i$. The discriminant is $\Delta= 13.6161$.
}
\end{figure}

\begin{figure}[!htpb]
\scalebox{0.5}{\includegraphics{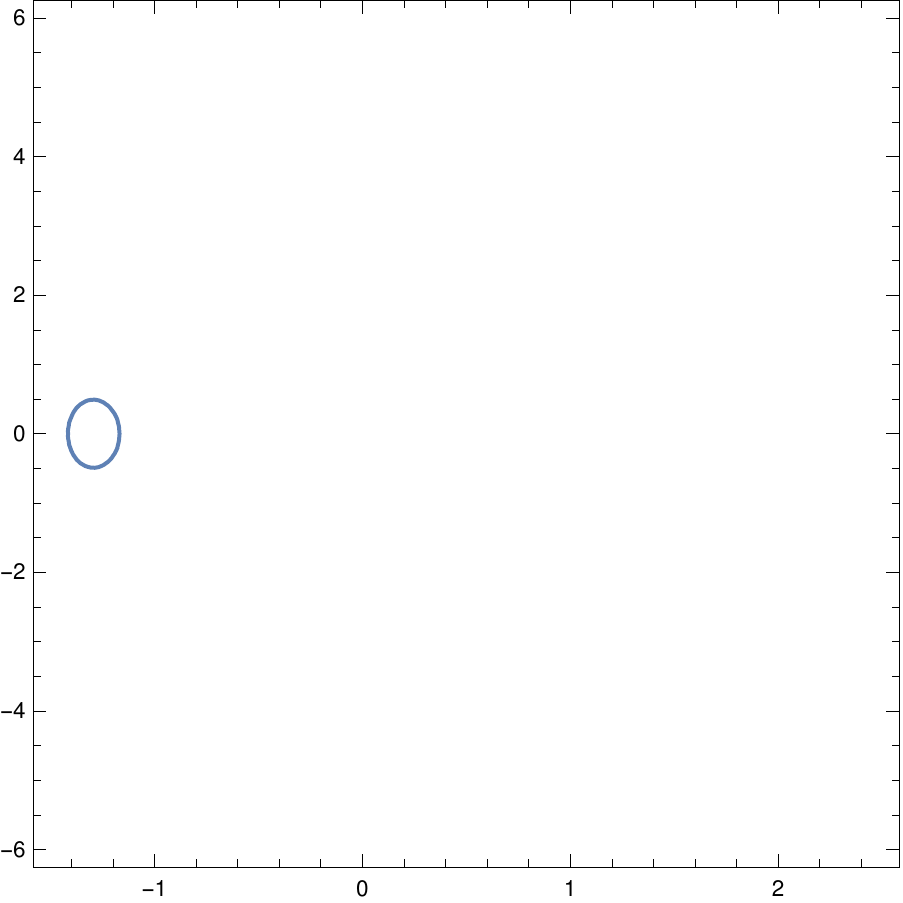}}
\scalebox{0.5}{\includegraphics{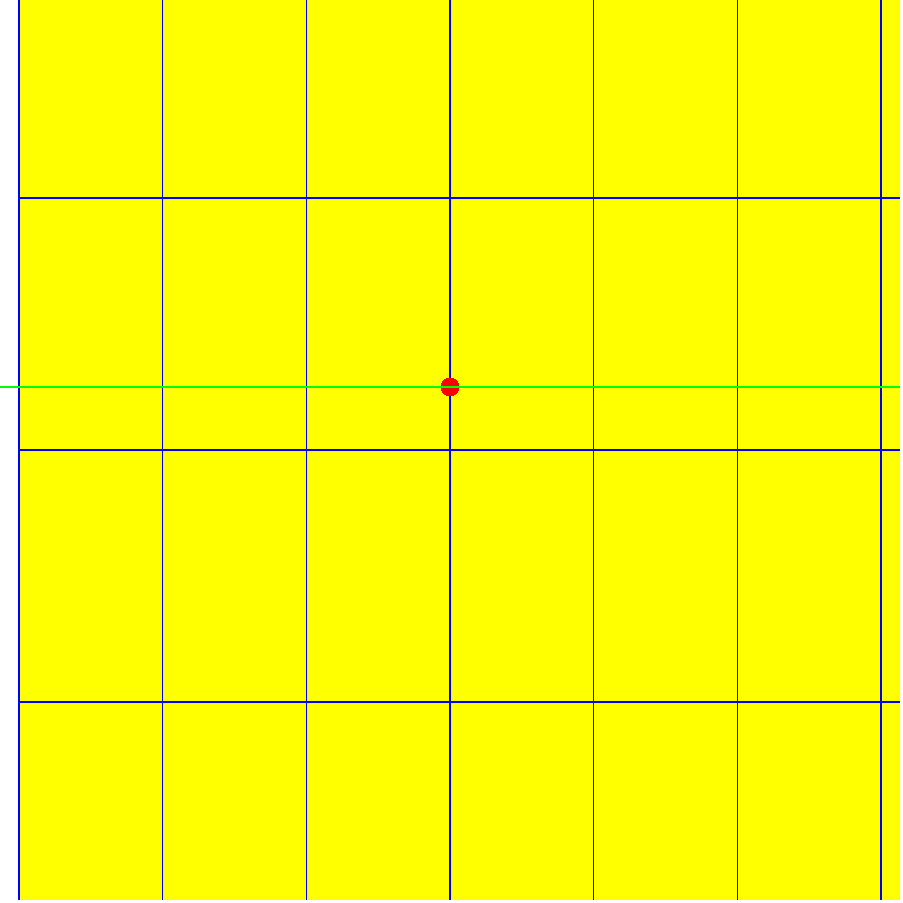}}
\caption{
For $c=-1$ and initial condition $\theta_0=0.$, and $\theta'_0=5.$,
we get the elliptic curve $y^2=4x^3-g_2 x-g_3$ with $g_2= 20.0833$,
and  $g_3=17.0787$.The lattice is given by $\omega_1= 1.59624\, +0. i$,
and  $\omega_2=0.\, +2.80121 i$. The initial point on the curve is $g_1=
0.\, +0.700302 i$. The discriminant is $\Delta= 225.$.
}
\end{figure}

\section*{Appendix: Mathematica Code}

\paragraph{}
Here is an implementation of the solution giving the Weierstrass elliptic
function in Mathematica:

\begin{tiny}
\lstset{language=Mathematica} \lstset{frameround=fttt}
\begin{lstlisting}[frame=single]
theta0=0.0; thetap0=1; c=-1; e=thetap0^2/2+4*c*Cos[theta0]; b=e/(6c);
a=2c-e^2/(6 c); w0=Exp[I*theta0]; wp0=I*w0*thetap0; w0=w0-b;
g2=-2a;    g3=-Chop[wp0^2 + 4*w0^3 + g2*w0]; P=WeierstrassP;
g1=Quiet[(z /. First[Solve[w0==c*P[c*z,{c*g2,g3}],z]]) /. C[1]->0 /. C[2]->0];
z[t_]:=b+c*P[c(t+g1),{c*g2,g3}]; theta[t_]:=Arg[z[t]]; 
Plot[theta[t], {t, 0, 10}, PlotRange->{-Pi,Pi}]
\end{lstlisting}
\end{tiny}

\paragraph{}
The solution with the Jacobi amplitude function does only plot the trajectory
for a half period: 

\begin{tiny}
\lstset{language=Mathematica} \lstset{frameround=fttt}
\begin{lstlisting}[frame=single]
c=-1; thetap0=1.0; 
s=DSolve[{x''[t]==4c*Sin[x[t]],x[0]==0,x'[0]==thetap0},x[t],t];
theta1[T_]:=First[x[t] /. s] /. t->T;
theta1[T_]:=2*JacobiAmplitude[thetap0*T/2,-16/(thetap0^2 c)]; 
Plot[theta1[t], {t, 0, 10}, PlotRange->{-Pi,Pi}]
\end{lstlisting}
\end{tiny}

\paragraph{}
A numerical solution given by the computer algebra system 
only gives good values for relatively 
short time intervals like $t=100$. For longer times, like 
$t=10^6$, the errors in the numerical methods have added up
significantly already and the solution has become unreliable.
The Weierstrass explicit solution is more convenient.

\begin{tiny}
\lstset{language=Mathematica} \lstset{frameround=fttt}
\begin{lstlisting}[frame=single]
c=-1; s=NDSolve[{x''[t]==4c*Sin[x[t]],x[0]==0,x'[0]==1},x[t],{t,0,10^6}];
theta2[T_] := First[x[t] /. s] /. t -> T;
theta[1.0*10^2] - theta2[1.0*10^2]
theta[1.0*10^6] - theta2[1.0*10^6]
\end{lstlisting}
\end{tiny}

\bibliographystyle{plain}

\end{document}